\documentclass[final]{siamltex}

\usepackage{algorithm}
\usepackage{algorithmic}
\usepackage{amsmath}
\usepackage{amsfonts}
\usepackage{multirow}
\usepackage{color}
\usepackage[top=1in, bottom=1in, left=1in, right=1in, dvips,letterpaper]{geometry}
\usepackage{setspace}
\newcommand\argmin{\mathop{\textrm{argmin}}}
\newcommand\crule[1][5cm]{%
  \par
  \nointerlineskip
  \centerline{\hbox to #1{\hrulefill}}%
  \nointerlineskip}

\numberwithin{equation}{section}
\numberwithin{algorithm}{section}
\newtheorem{thm}{{\sc Theorem}}[section]
\newtheorem{lem}{{\sc Lemma}}[section]
\newtheorem{cor}[thm]{Corollary}
\newtheorem{rem}{Remark}[section]
\newtheorem{prop}{Proposition}[section]
\newtheorem{defi}{{\sc Definition}}[section]

\newcommand{\R}{\mathbb{R}}
\newcommand{\E}{\mathbb{E}}
\newcommand{\SFO}{\mathcal{SFO}}

\newcommand{\tr}{^{\sf T}}

\newcommand{\C}[1]{{\cal {#1}}}

\newcommand{\be}{\begin{equation}}
\newcommand{\ee}{\end{equation}}

\newcommand{\etal}{{\it et al.\ }}

\title{Stochastic Quasi-Newton Methods for \\ Nonconvex Stochastic Optimization}

\author{Xiao Wang
\thanks{School of Mathematical Sciences, University of Chinese Academy of Sciences, China. Email: wangxiao@ucas.ac.cn. Research of this author was supported in part by Postdoc Grant 119103S175, UCAS President Grant Y35101AY00 and NSFC Grant 11301505.}
\and Shiqian Ma
\thanks{Corresponding author. Department of Systems Engineering and Engineering Management, The Chinese University of Hong Kong, Shatin, N. T., Hong Kong, China. Email: sqma@se.cuhk.edu.hk. Research of this author was supported in part by the Hong Kong Research Grants Council General Research Fund (Grant 14205314).}
\and Wei Liu
\thanks{IBM T. J. Watson Research Center, Yorktown Heights, New York, USA. Email: weiliu@us.ibm.com}
}

\begin{document}

\maketitle

\begin{center} December 1, 2014 \end{center}

\begin{abstract}
In this paper we study stochastic quasi-Newton methods for nonconvex stochastic optimization, where we assume that only stochastic information of the gradients of the objective function is available via a stochastic first-order oracle ($\SFO$). Firstly, we propose a general framework of stochastic quasi-Newton methods for solving nonconvex stochastic optimization. The proposed framework extends the classic quasi-Newton methods working in deterministic settings to stochastic settings, and we prove its almost sure convergence to stationary points. Secondly, we propose a general framework for a class of randomized stochastic quasi-Newton methods, in which the number of iterations conducted by the algorithm is a random variable. The worst-case $\SFO$-calls complexities of this class of methods are analyzed. Thirdly, we present two specific methods that fall into this framework, namely stochastic damped-BFGS method and stochastic cyclic Barzilai-Borwein method. Finally, we report numerical results to demonstrate the efficiency of the proposed methods.

\vspace{0.8cm}

\noindent {\bf Keywords:} Stochastic Optimization, Nonconvex Optimization, Stochastic Approximation, Quasi-Newton Method, BFGS Method, Barzilai-Borwein Method, Complexity

\vspace{0.5cm}

\noindent {\bf Mathematics Subject Classification 2010:} 90C15; 90C30; 62L20; 90C60

\end{abstract}



\section{Introduction}\label{sec:intro}

In this paper, we consider the following stochastic optimization problem:
\begin{equation}\label{orig-prob}
\min_{x\in\R^n}\quad  f(x)
\end{equation}
where $f:\R^n \to \R$ is continuously differentiable and possibly nonconvex. We assume that the exact information of function values and gradients of $f$ are not available and only noisy gradients of $f$ can be obtained via subsequent calls to a {\it stochastic first-order oracle} ($\SFO$). Problem \eqref{orig-prob} arises in many applications, including machine learning \cite{mbps09}, simulation-based optimization \cite{f02}, and mixed logit modeling problems in economics and transportation \cite{bbt00,bct06,hg03}. In these applications,  the objective function is sometimes given in the form of an expectation of certain function with a random variable being a parameter:
\[
f(x)=\E[F(x,\xi)],  \quad \mbox{or} \quad f(x) = \int_{\Xi} F(x,\xi)dP(\xi),
\]
where $\xi$ denotes a random variable and its distribution $P$ is supported on $\Xi$. Since in many cases either the integral is difficult to evaluate or function $F(\cdot,\xi)$ is not given explicitly, the function values and gradients of $f$ cannot be easily obtained and only noisy gradient information of $f$ is available.

The idea of employing stochastic approximation (SA) to solve stochastic programming problems can be traced back to the seminal work by Robbins and Monro \cite{rm51}. The classical SA method mimics the steepest gradient descent method using a stochastic gradient, i.e., it updates the iterates via
\[
x_{k+1} = x_k - \alpha_k G_k,
\]
where $G_k$ is an unbiased estimate of the gradient of $f$ at $x_k$, and $\alpha_k$ is a stepsize for the stochastic gradient step. In the literature, the SA method is also referred to as stochastic gradient descent (SGD) method.
The SA method has been further studied extensively in \cite{C54,e83,g78,p90,pj92,rs86,s58}, and the main focus in these papers has been the convergence of SA in different settings. Recently, there have been lots of interests in analyzing
the worst-case complexity of SA methods. Works along this direction were mainly ignited by the complexity theory developed by Nesterov for first-order methods engaging in solving convex optimization problems \cite{ny83,NesterovConvexBook2004}.
Nemirovski \etal \cite{njls09} proposed a mirror descent SA method for solving nonsmooth convex stochastic programming problem $x^*:=\argmin \{ f(x)\mid x\in X\}$ and analyzed its worst-case iteration complexity, where $f$ is nonsmooth and convex and $X$ is a convex set. Specifically, it was shown in \cite{njls09} that for any given $\epsilon>0$, the proposed mirror descent SA method needs $O(\epsilon^{-2})$ iterations to obtain an $\bar{x}$ such that $\E[f(\bar{x})-f(x^*)]\le\epsilon$,
where $\E[y]$ denotes the expectation of the random variable $y$. Other SA methods with provable complexity analysis for solving convex stochastic optimization problems have also been studied in \cite{gl12,jntv05,jrt08,l12,lns12}.

It is noted that the SA methods mentioned above all concentrated on convex stochastic optimization problems. Recently there have been lots of interests on SA methods for nonconvex stochastic optimization problems \eqref{orig-prob} in which $f$ is a nonconvex function. Ghadimi and Lan \cite{gl13} proposed a randomized stochastic gradient (RSG) method for nonconvex stochastic optimization \eqref{orig-prob}. RSG returns an iterate from a randomly chosen iteration as an approximate solution. It is shown in \cite{gl13} that to return an $\epsilon$-solution $\bar{x}$, i.e., $\E[\|\nabla f(\bar{x})\|^2]\le\epsilon$, the total number of $\SFO$-calls needed by RSG is in the order of $O(\epsilon^{-2})$. Ghadimi and Lan \cite{gl132} also studied an accelerated stochastic SA method for solving stochastic optimization problems \eqref{orig-prob} based on Nesterov's accelerated gradient method, which improved the complexity for convex cases from $O(\epsilon^{-2})$ to $O(\epsilon^{-4/3})$. A class of nonconvex stochastic optimization problems, in which the objective function is a composition of a nonconvex function $f$ and a convex nonsmooth function $g$, i.e., $x^*:=\argmin\{f(x)+g(x):x\in\R^n\}$,
was considered by Ghadimi \etal in \cite{glz13}, and a mini-batch SA method was proposed and its worst-case $\SFO$-calls complexity was analyzed. In \cite{dl13}, a stochastic block mirror descent method, which incorporates the block-coordinate decomposition scheme into  stochastic mirror-descent methodology, was proposed for a nonconvex stochastic optimization problem $x^*=\argmin\{f(x):x\in X\}$ with $X$ having a block structure. More recently, Wang \etal \cite{WangMaYuan13} proposed a penalty method for nonconvex stochastic optimization problems with nonlinear constraints, and also analyzed its $\SFO$-calls complexity.

The aforementioned methods are all first-order methods in the sense that they only use (stochastic) first-order derivative information of the objective function. In this paper, we consider methods for solving \eqref{orig-prob} that employ certain approximate second-order derivative information of the objective function. Since approximate second-order information is used, this kind of methods are expected to take less number of iterations to converge, with the price that the per-iteration computational effort is possibly increased. Along this line, there have been some works in designing stochastic quasi-Newton methods for unconstrained stochastic optimization problems.
Methods of this type usually employ the following updates
\be\label{second-order-update}
x_{k+1} = x_k - \alpha_k B_k^{-1} G_k, \quad \mbox{or}\quad x_{k+1} = x_k - \alpha_k H_k G_k,
\ee
where $B_k$ (resp. $H_k$) is a positive definite matrix that approximates the Hessian matrix (resp. inverse of the Hessian matrix) of $f(x)$ at $x_k$.
Some representative works in this class of methods are discussed in the following. Among the various SGD methods, the adaptive subgradient (AdaGrad) proposed in \cite{DHS2011} has been proven to be quite efficient in practice. AdaGrad takes the form of \eqref{second-order-update} with $B_k$ being a diagonal matrix that estimates the diagonal of the squared root of the uncentered covariance matrix of the gradients. \cite{BBG-09} also studied the method using SGD with a diagonal rescaling matrix based on the secant condition associated with quasi-Newton methods. In addition, it was shown in \cite{BBG-09} that if $B_k$ is chosen as the exact Hessian at the optimal solution $x^*$, the number of iterations needed to achieve an $\epsilon$-solution $\bar{x}$, i.e., $\E[f(\bar{x})-f(x^*)]\le\epsilon$, is in the order of $O(\epsilon^{-1})$. However, the optimal solution of the problem usually cannot be obtained beforehand, so the exact Hessian information remains unknown. \cite{RF-2010} discussed the necessity of including both Hessian and covariance matrix information in a (stochastic) Newton type method.
The quasi-Newton method proposed in \cite{BCNN} uses some subsampled Hessian algorithms via the sample average approximation (SAA) approach to estimate Hessian-vector multiplications.
In \cite{BHNS-14}, the authors proposed to use the SA approach instead of SAA to estimate the curvature information. This stochastic quasi-Newton method is based on L-BFGS \cite{Liu-Nocedal-89} and performs very well in some problems arising from machine learning, but no theoretical convergence analysis was provided in \cite{BHNS-14}. Stochastic quasi-Newton methods based on BFGS and L-BFGS updates were also studied for online convex optimization in Schraudolph \etal \cite{SchYuGue07}, with no convergence analysis provided, either. Mokhtari and Ribeiro \cite{mr10} propose a regularized stochastic BFGS method (RES) for solving \eqref{orig-prob} with $f$ being strongly convex, and proved its almost sure convergence. Recently, Mokhtari and Ribeiro \cite{mr14-2} proposed an online L-BFGS method that is suitable for strongly convex stochastic optimization problems arising in the regime of large scale machine learning, and analyzed its global convergence. It should be noted that all the aforementioned methods based on stochastic quasi-Newton information mainly focused on solving convex or strongly convex stochastic optimization problems.

As discovered by several groups of researchers \cite{BBG-09,mr10,SchYuGue07}, when solving convex stochastic optimization problems, stochastic quasi-Newton methods may result in nearly singular Hessian approximations $B_k$ due to the presence of stochastic information. \cite{mr10} proposed a regularized BFGS update strategy which can preserve the positive-definiteness of $B_k$. However, for nonconvex optimization problems, preserving the positive-definiteness of $B_k$ is difficult even in deterministic settings. In classic quasi-Newton methods for nonconvex deterministic optimization, line search techniques are usually incorporated to guarantee the positive-definiteness of $B_k$. However, performing the line search techniques in stochastic optimization is no longer practical, because the exact function values are not available.
Therefore, a crucial issue in applying quasi-Newton methods to solve nonconvex stochastic optimization \eqref{orig-prob} is how to keep the positive-definiteness of the updates $B_k$ without using the line search techniques. In this paper, we will discuss and address this issue. Our contributions in this paper are as follows.
\begin{itemize}
\item[1.] We propose a general framework of stochastic quasi-Newton methods for solving nonconvex stochastic optimization problem \eqref{orig-prob}. In addition, we analyze its almost sure convergence to the stationary point of \eqref{orig-prob}.
\item[2.] We propose a general framework of randomized stochastic quasi-Newton methods for solving \eqref{orig-prob}. In this kind of methods, the methods return an iterate from a randomly chosen iteration. We analyze their worst-case $\SFO$-calls complexity to find an $\epsilon$-solution $\bar{x}$, i.e., $\E[\|\nabla f(\bar{x})\|^2]\le \epsilon$.
\item[3.] We propose two concrete stochastic quasi-Newton update strategies, namely stochastic damped-BFGS update and stochastic cyclic-BB-like update, to adaptively generate positive definite Hessian approximations. Both strategies fit into the proposed general frameworks, so the established convergence and complexity results apply directly.
\end{itemize}

{\bf Notation.} The gradient of $f(x)$ is denoted as $\nabla f(x)$. The subscript $_k$ refers to the iteration number in an algorithm, e.g., $x_k$ is the $k$-th $x$ iterate. Without specification, $\|x\|$ represents the Euclidean norm of vector $x$. Both $\langle x,y \rangle$ and $x\tr y$ with $x,y\in\R^n$ denote the Euclidean inner product of $x$ and $y$. $\lambda_{\max}(A)$ denotes the largest eigenvalue of a symmetric matrix $A$. $A\succeq B$ with $A,B\in\R^{n\times n}$ means that $A-B$ is positive semidefinite. In addition, $\mod(a,b)$ with two positive integers $a$ and $b$ denotes the modulus of division $a/b$. We also denote by $\mathrm{P}_{\Omega}$ the projection onto a closed convex set $\Omega$.

{\bf Organization.} The rest of this paper is organized as follows. In Section \ref{sec:RSSA}, we present a general framework of stochastic quasi-Newton methods for nonconvex stochastic optimization \eqref{orig-prob} and analyze its convergence in expectation. In Section \ref{sec:rssa}, we present a general framework of randomized stochastic quasi-Newton methods and analyze its worst-case $\SFO$-calls complexity for returning an $\epsilon$-solution. In Section \ref{sec:SDBFGS}, we propose two concrete quasi-Newton update strategies, namely stochastic damped-BFGS update and stochastic cyclic-BB-like update. In Section \ref{sec:num} we report some numerical experimental results. Finally, we draw our conclusions in Section \ref{sec:conclusions}.


\section{A general framework for nonconvex stochastic quasi-Newton methods}\label{sec:RSSA}

In this section we study the stochastic quasi-Newton methods for nonconvex stochastic optimization problem \eqref{orig-prob}.
We assume that only noisy gradient information of $f$ is available via $\SFO$ calls. Namely, for the input $x$, $\SFO$ will output a stochastic gradient $G(x,\xi)$ of $f$, where $\xi$ is a random variable whose distribution is supported on $\Xi\subseteq \R^d$. Here we assume that $\Xi$ does not depend on $x$.

We now give some assumptions required throughout this paper.
\begin{itemize}
\item[\textbf{AS.1}]\quad $f\in\mathcal{C}^1(\R^n)$, i.e., $f: \R^n\to \R$ is continuously differentiable. $f(x)$ is lower bounded by $f^{low}$ for any $x\in\R^n$. $\nabla f$ is globally Lipschitz continuous with Lipschitz constant $L$.
\item[\textbf{AS.2}]\quad For any iteration $k$, we have
\begin{align}\label{c}
  a)\quad & \E_{\xi_k}\left[G(x_k,\xi_k)\right]=\nabla f(x_k),\\
  b)\quad & \E_{\xi_k}\left[\|G(x_k,\xi_k)-\nabla f(x_k)\|^2\right]\le \sigma^2,
\end{align}
where $\sigma>0$ is the noise level of the gradient estimation, and $\xi_k, k=1,\ldots,$ are independent to each other, and they are also assumed to be independent of $x_k$.
\end{itemize}

In SGD methods, iterates are normally updated through
\be \label{sgd}
x_{k+1} = x_k - \alpha_k G(x_k,\xi_k),
\ee
or the following mini-batch version
\be \label{sgd2}
x_{k+1} = x_k -\alpha_k\cdot\frac{1}{m_k}\sum_{i=1}^{m_k} G(x_k,\xi_{k,i}),
\ee
where $m_k\ge 1$ is a positive integer and refers to the {\it batch size} in the $k$-th iteration. For deterministic unconstrained optimization, quasi-Newton methods have been proven to perform better convergence speed than gradient-type methods, because approximate second-order derivative information is employed. In deterministic unconstrained optimization, quasi-Newton methods update the iterates using
\be\label{quasi}
x_{k+1} = x_k - \alpha_k B_k^{-1}\nabla f(x_k),
\ee
where the stepsize $\alpha_k$ is usually determined by line search techniques, and $B_k$ is a positive definite matrix that approximates the Hessian matrix of $f(x)$ at iterate $x_k$. One widely-used updating strategy for $B_k$ is the following BFGS formula \cite{nw06}:   \be\label{bfgs}
\begin{aligned}
(\mathrm{BFGS}): \quad B_{k+1} = B_k + \frac{y_ky_k\tr }{s_k\tr y_k} - \frac{B_ks_ks_k\tr B_k}{s_k\tr B_ks_k},
\end{aligned}
\ee
where $s_k := x_{k+1}-x_k$ and $y_k:=\nabla f(x_{k+1}) - \nabla f(x_k)$. It is known that \eqref{bfgs} preserves the positive-definiteness of sequence $\{B_k\}$.
BFGS method and the limited memory BFGS method \cite{Liu-Nocedal-89} demonstrate faster convergence speed than gradient methods both theoretically and numerically. Interested readers are referred to \cite{nw06} for more details on quasi-Newton methods in deterministic settings.

In the stochastic settings, since the exact gradients of $f$ are not available, the update formula \eqref{bfgs} cannot guarantee that $B_{k+1}$ is positive definite. To overcome this difficulty, Mokhtari and Ribeiro \cite{mr10} proposed the following updating formula which preserves the positive-definiteness of $B_k$:
\be \label{x-k}
x_{k+1} = x_k - \alpha_k(B_k^{-1}+\zeta_k I)\,G_k,
\ee
where $\zeta_k$ is a safeguard parameter such that $B_k^{-1} +\zeta_k I$ is uniformly positive definite for all $k$, and $G_k$ is defined as
\begin{equation}\label{G-k}
    G_k = \frac{1}{m_k}\sum_{i=1}^{m_k}G(x_k,\xi_{k,i}),
 \end{equation}
 where the positive integer $m_k$ denotes the batch size in gradient samplings.
From AS.2 we know that $G_k$ has the following properties:
\be\label{Exp-G-k}
\E[G_k|x_k] = \nabla f(x_k), \quad \E[\|G_k-\nabla f(x_k)\|^2|x_k]\le \frac{\sigma^2}{m_k}.
\ee
We also make the following bound assumption on $B_k$. Note that similar assumption was required in \cite{mr10}.
\begin{itemize}
\item[{\bf AS.3}]\quad There exist two positive scalars $m$ and $M$ such that
\[
m I \preceq B_k^{-1}+\zeta_k I \preceq MI, \quad \mbox{ for any } k,
\]
where $m$ and $M$ are positive scalars.
\end{itemize}
From \eqref{G-k}, it follows that $G_k$ depends on random variables $\xi_{k,1},\ldots,\xi_{k,m_k}$. We denote $\xi_k:=(\xi_{k,1},\ldots,\xi_{k,m_k})$. We use $\xi_{[k]}$ to denote the collection of all the random variables in the first $k$ iterations, i.e., $\xi_{[k]}:=(\xi_1,\ldots,\xi_k)$.
It is easy to see from \eqref{G-k} and \eqref{x-k} that the random variable $x_{k+1}$ depends on $\xi_{[k]}$ only. Since $B_k$ depends on $x_k$, we make the following assumption on $B_k (k\ge 2)$ (note that $B_1$ is pre-given in the initial setting):
\begin{itemize}
\item[{\bf AS.4}]\quad For any $k\ge 2$, the random variable $B_k$ depends only on $\xi_{[k-1]}$.
\end{itemize}

The following equality follows directly from {\bf AS.4} and \eqref{Exp-G-k}:
\[
\E[B_k^{-1}G_k|\xi_{[k-1]}] = B_k^{-1}\nabla f(x_k).
\]
We will see later that this equality plays a key role in analyzing our stochastic quasi-Newton methods. Moreover, both assumptions {\bf AS.3} and {\bf AS.4} can be realized and we will propose two specific updating schemes of $B_k$ that satisfy {\bf AS.3-4} in Section \ref{sec:SDBFGS}.

We now present a general framework of stochastic quasi-Newton methods (SQN) for solving \eqref{orig-prob} in Algorithm \ref{sso-uncons}.

\begin{algorithm}[ht]
\caption{{\bf SQN: Stochastic quasi-Newton method for nonconvex stochastic optimization \eqref{orig-prob}}}
\label{sso-uncons}
\begin{algorithmic}[1]
\REQUIRE {Given $x_1\in\R^n$, a positive definite matrix $B_1\in\R^{n\times n}$, batch sizes $\{m_k\}_{k\ge 1}$, safeguard parameters $\{\zeta_k\}_{k\ge1}$ and stepsizes $\{\alpha_k\}_{k\ge 1}$ satisfying
    \be\label{alpha-inf}
        \sum_{i=0}^{+\infty}\alpha_i = +\infty, \qquad \sum_{i=0}^{+\infty}\alpha_i^2 < +\infty.
    \ee
}
\FOR {$k=1,2,\ldots$}
    \STATE    Calculate $G_k$ through \eqref{G-k}, i.e.,
        \begin{equation*}
            G_k = \frac{1}{m_k}\sum_{i=1}^{m_k}G(x_k,\xi_{k,i}).
        \end{equation*}
    \STATE Calculate $x_{k+1}$ through \eqref{x-k}, i.e.,
        \[
        x_{k+1} = x_k - \alpha_k(B_k^{-1}+\zeta_k I)\,G_k.
        \]
    \STATE Generate $B_{k+1}$ that satisfies assumptions {\bf AS.3} and {\bf AS.4}.
\ENDFOR
\end{algorithmic}
\end{algorithm}

We now analyze the convergence of Algorithm \ref{sso-uncons}. Note that if the sequence of iterates $\{x_k\}$ generated by Algorithm \ref{sso-uncons} lies in a compact set, then it follows from {\bf AS.1} that 
$\{\nabla f(x_k)\}$ is bounded. Then there exists $\bar{M}>0$ such that
\be\label{bound}
\|\nabla f(x_k)\|\le \bar{M}.
\ee

The following lemma provides a descent property of the objective value of Algorithm \ref{sso-uncons}.
\begin{lem}\label{lem3.1}
Assume that $\{x_k\}$ is generated by Algorithm \ref{sso-uncons} and \eqref{bound} and assumptions {\bf AS.1-4} hold. Then the expectation of function value $f(x_{k+1})$ conditioned on $x_k$ satisfies
\be \label{exp-red}
\E[f(x_{k+1})|x_k] \le f(x_k) - \alpha_k m\|\nabla f(x_k)\|^2 + \frac{1}{2}L\alpha_k^2M^2\left(\bar{M}^2+\frac{\sigma^2}{m_k}\right),
\ee
where the conditioned expectation is taken with respect to $\xi_k$.
\end{lem}

\begin{proof}
Using {\bf AS.1}, {\bf AS.3} and \eqref{x-k}, we have
\begin{align}
&\, f(x_{k+1}) \notag \\
 \le & \,f(x_k) + \langle \nabla f(x_k), x_{k+1} - x_k \rangle + \frac{L}{2}\|x_{k+1}-x_{k}\|^2  \notag \\
 = & \,f(x_k) -\alpha_k\langle \nabla f(x_k), (B_k^{-1}+\zeta_k I)G_k \rangle + \frac{L}{2}\alpha_k^2\|(B_k^{-1}+\zeta_k I)G_k\|^2  \notag \\
 \leq &\, f(x_k)  -\alpha_k\langle \nabla f(x_k), (B_k^{-1}+\zeta_k I)\nabla f(x_k)\rangle - \alpha_k\langle \nabla f(x_k), (B_k^{-1}+\zeta_k I)\delta_k\rangle + \frac{L}{2}\alpha_k^2M^2\|G_k\|^2, \label{red-ori}
\end{align}
where $\delta_k = G_k - \nabla f(x_k)$.
Taking expectation on both sides of \eqref{red-ori} conditioned on $x_k$ with respect to $\xi_k$ and noticing that $\E[\delta_k|x_k]=0$,
we obtain from {\bf AS.4} that
\be \label{red}
\E[f(x_{k+1})|x_k]\le  f(x_k) -\alpha_k\langle \nabla f(x_k), (B_k^{-1}+\zeta_k I)\nabla f(x_k) \rangle + \frac{L}{2}\alpha_k^2M^2\E[\|G_k\|^2|x_k].
\ee
From \eqref{Exp-G-k}, \eqref{bound} and $\E[\delta_k|x_k]=0$, we have the following relations:
\begin{align*}
\E[\|G_k\|^2|x_k]& =  \E[\|G_k-\nabla f(x_k) + \nabla f(x_k)\|^2|x_k] \\
                 & =  \E[\|\nabla f(x_k)\|^2|x_k] + \E[\|G_k - \nabla f(x_k)\|^2|x_k] +  2\E[\delta_k,\nabla f(x_k)\rangle |x_k] \\
                 & = \|\nabla f(x_k)\|^2 + \E[\|G_k - \nabla f(x_k)\|^2|x_k]  \\
                 & \le \bar{M}^2 + \frac{\sigma^2}{m_k},
\end{align*}
which together with \eqref{red} and {\bf AS.3} yields \eqref{exp-red}.
\end{proof}

Before proceeding our analysis, we introduce the definition of {\it supermartingale} (see \cite{Durrett-10} for more details).
\begin{defi}\label{def2.1}
Let $\mathcal{F}_k$ be an increasing sequence of $\sigma$-algebra. If $\{X_k\}$ is a stochastic process satisfying
\begin{itemize}
\item[(i)] $\E[|X_k|]<\infty$;
\item[(ii)] $X_k\in\mathcal{F}_k$ for all $k$;
\item[(iii)] $\E[X_{k+1}|\mathcal{F}_k]\le X_k$ for all $k$,
\end{itemize}
then $\{X_k\}$ is said to be a supermartingale.
\end{defi}

The following theorem states the convergence of a nonnegative supermartingale (see, e.g., Theorem 5.2.9 in \cite{Durrett-10}).
\begin{prop}\label{prop2.1}
If $\{X_k\}$ is a nonnegative supermartingale, then $\lim_{k\to\infty}X_k\to X$ almost surely and $\E[X]\le\E[X_0]$.
\end{prop}

Now we are ready to give the main convergence result of our stochastic quasi-Newton method (Algorithm \ref{sso-uncons}). Its proof essentially follows Theorem 1 in \cite{mr10}, but our assumptions here are relatively weaker.
\begin{thm}\label{thm3.1}
Assume that \eqref{bound} and assumptions {\bf AS.1-4} hold for $\{x_k\}$ generated by Algorithm \ref{sso-uncons} with batch size $m_k = \bar{m}$ for any $k$. Then
\be \label{liminf_g}
\liminf_{k\to\infty}\, \|\nabla f(x_k)\| = 0, \quad \mbox{ with probability 1}.
\ee
\end{thm}
\begin{proof}
Define
\begin{align*}
\gamma_k &:= f(x_k)+ \frac{LM^2(\bar{M}^2+\sigma^2/\bar{m})}{2}\sum_{i=k}^\infty \alpha_i^2, \\
\beta_k &:= \alpha_k m\|\nabla f(x_k)\|^2.
\end{align*}
Let $\C{F}_k$ be the $\sigma$-algebra measuring $\gamma_k$, $\beta_k$ and $x_k$. Then from \eqref{exp-red} we have that
\begin{align}
\E[\gamma_{k+1}|\C{F}_k] & = \E[f(x_{k+1})|\C{F}_k] + \frac{LM^2(\bar{M}^2+\sigma^2/\bar{m})}{2}\sum_{i=k+1}^\infty \alpha_i^2 \notag \\
                         & \le f(x_k) -\alpha_k m\|\nabla f(x_k)\|^2 + \frac{LM^2(\bar{M}^2+\sigma^2/\bar{m})}{2}\sum_{i=k}^\infty \alpha_i^2  \notag \\
                         & =\gamma_k - \beta_k, \label{diff}
\end{align}
which implies that
\[
\E[\gamma_{k+1}-f^{low}|\C{F}_k] \le \gamma_k - f^{low} - \beta_k.
\]
Since $\beta_k\ge0$, we have $0\le\E[\gamma_k-f^{low}]\le \gamma_1-f^{low}<+\infty$. Then according to Definition \ref{def2.1}, $\{\gamma_k-f^{low}\}$ is a supermartingale. Therefore, Proposition \ref{prop2.1} shows that there exists $\gamma$ such that $\lim_{k\to\infty}\gamma_k=\gamma$ with probability 1, and $\E[\gamma]\le \E[\gamma_1]$.
Note that from \eqref{diff} we have $\E[\beta_k] \le \E[\gamma_k]-\E[\gamma_{k+1}]$. Thus,
\[
\E[\sum_{k=0}^{\infty}\beta_k] \le \sum_{k=0}^{\infty}(\E[\gamma_{k}]-\E[\gamma_{k+1}]) < +\infty,
\]
which further yields that
\[
\sum_{k=0}^{\infty}\beta_k = m\sum_{k=0}^\infty \alpha_k\|\nabla f(x_k)\|^2 < +\infty \quad \mbox{  with probability 1}.
\]
Since $\sum_{k=0}^{\infty}\alpha_k =+\infty$, it follows that \eqref{liminf_g} holds.
\end{proof}

\begin{rem}
Note that in Algorithm \ref{sso-uncons} we require that the stepsizes $\alpha_k$ satisfy \eqref{alpha-inf}. This condition is easy to be satisfied. For example, one very simple strategy is to set $\alpha_k=O(1/k)$. In the numerical experiments we will show later, we test the performance of the algorithm using different settings of $\alpha_k$ that satisfy \eqref{alpha-inf}.
\end{rem}

\section{A general framework for randomized stochastic quasi-Newton method for \eqref{orig-prob}}\label{sec:rssa}

In Section \ref{sec:RSSA}, we proposed a general framework for stochastic quasi-Newton methods and studied its convergence. In this section, we propose another algorithmic framework, which is called randomized stochastic quasi-Newton method (RSQN), for solving \eqref{orig-prob}. RSQN is very similar to SQN (Algorithm \ref{sso-uncons}), with the only difference being that RSQN returns the iterate from a randomly chosen iteration as the final approximate solution. The idea of returning the iterate from a randomly chosen iteration is inspired by the RSG method \cite{gl13}. It is shown in \cite{gl13} that by randomly choosing an iteration number $R$, RSG returns $x_R$ as an $\epsilon$-solution, i.e., $\E[\|\nabla f(x_R)\|^2]\le \epsilon$ with the worst-case $\SFO$-calls complexity being $O(\epsilon^{-2})$. Inspired by RSG, we propose the following RSQN (Algorithm \ref{finit-sso-uncons}) and analyze its worst-case $\SFO$-calls complexity.

\begin{algorithm}[ht]
\caption{{\bf RSQN: Randomized stochastic quasi-Newton method for nonconvex stochastic optimization \eqref{orig-prob}}}
\label{finit-sso-uncons}
\begin{algorithmic}[1]
\REQUIRE {Given maximum iteration number $N$, $x_1\in\R^n$, a positive definite matrix $B_1\in\R^{n\times n}$, stepsizes $\{\alpha_k\}_{k\ge 1}$ , batch sizes $\{m_k\}_{k\ge 1}$ and positive safeguard parameters $\{\zeta_k\}_{k\ge 1}$. Randomly choose $R$ according to probability mass function $P_R$ supported on $\{1,\ldots, N\}$.}
\ENSURE {$x_R$}.
\FOR{$k=1,2,\ldots, R$}
\STATE Calculate $G_k$ through \eqref{G-k}, i.e., \[G_k = \frac{1}{m_k}\sum_{i=1}^{m_k}G(x_k,\xi_{k,i}).\]
\STATE Calculate $x_{k+1}$ through \eqref{x-k}, i.e., \[x_{k+1} = x_k - \alpha_k(B_k^{-1}+\zeta_k I)\,G_k.\]
\STATE Generate $B_{k+1}$ such that assumptions {\bf AS.3} and {\bf AS.4} hold.
\ENDFOR
\end{algorithmic}
\end{algorithm}

In the following, we give the worst-case $\SFO$-calls complexity of Algorithm \ref{finit-sso-uncons} for returning $x_R$ such that $\E[\|\nabla f(x_R)\|^2]\le\epsilon$.

\begin{thm} \label{thm3.2}
Assume assumptions {\bf AS.1-4} hold, and the stepsizes $\alpha_k$ in Algorithm \ref{finit-sso-uncons} are chosen such that $0<\alpha_k\le 2m/(LM^2)$ with $\alpha_k<2m/(LM^2)$ for at least one $k$. Moreover, suppose that the probability mass function $P_R$ is given as follows:
\be \label{P_R}
P_R(k):=\mathrm{Prob}\{R=k\}= \frac{ m\alpha_k - L M^2\alpha_k^2/2}{\sum_{k=1}^N
\left ( m\alpha_k - L M^2\alpha_k^2/2\right)}, \quad k=1,\ldots,N.
\ee
Then for any $N\ge1$, we have
\be \label{exp-gra}
\E[\|\nabla f(x_R)\|^2]\le \frac{D_f + (LM^2\sigma^2)/2\sum_{k=1}^N(\alpha_k^2/m_k)}{\sum_{k=1}^N
\left ( m\alpha_k - L M^2\alpha_k^2/2\right)},
\ee
where $D_f:=f(x_1)-f^{low}$ and the expectation is taken with respect to $R$ and $\xi_{[N]}$.
\end{thm}
\begin{proof}
From \eqref{red-ori} it follows that
\begin{align*}
f(x_{k+1})  \le &\, f(x_k)  -\alpha_k\langle \nabla f(x_k), (B_k^{-1}+\zeta_k I)\nabla f(x_k)\rangle - \alpha_k\langle \nabla f(x_k), (B_k^{-1}+\zeta_k I)\delta_k\rangle + \\
&\, \frac{L}{2}\alpha_k^2M^2 [\|\nabla f(x_k)\|^2 + 2\langle \nabla f(x_k), \delta_k\rangle + \|\delta_k\|^2]\\
\le & \,f(x_k) - \left( m\alpha_k - \frac{L M^2}{2}\alpha_k^2\right) \|\nabla f(x_k)\|^2 + \frac{L M^2}{2}\alpha_k^2\|\delta_k\|^2 + LM^2\alpha_k^2 \langle \nabla f(x_k),\delta_k\rangle \\
&\, - \alpha_k\langle \nabla f(x_k), (B_k^{-1}+\zeta_k I)^{-1}\delta_k \rangle,
\end{align*}
where $\delta_k=G_k-\nabla f(x_k)$. Summing up the above inequality over $k=1,\ldots,N$ and noticing that $\alpha_k\le 2m/(LM^2)$, we have
\begin{align}
&\sum_{k=1}^N
\left ( m\alpha_k - \frac{L M^2}{2}\alpha_k^2\right) \|\nabla f(x_k)\|^2 \notag\\
\le & f(x_1) - f^{low} + \frac{LM^2}{2}\sum_{k=1}^N\alpha_k^2\|\delta_k\|^2 + \sum_{k=1}^N(LM^2\alpha_k^2 \langle \nabla f(x_k),\delta_k\rangle - \alpha_k\langle \nabla f(x_k), (B_k^{-1}+\zeta_k I)^{-1}\delta_k \rangle).  \label{sum-gra}
\end{align}
Notice that both $x_k$ and $B_k$ depend only on $\xi_{[k-1]}$. Thus, by {\bf AS.2} and {\bf AS.4} we have that
\[
\E_{\xi_k}[\langle \nabla f(x_k), \delta_k \rangle|\xi_{[k-1]}]=0, \qquad \E_{\xi_k}[\langle \nabla f(x_k), (B_k^{-1}+\zeta_k I)\delta_k \rangle|\xi_{[k-1]}]=0.
\]
Moreover, from \eqref{Exp-G-k} it follows that $\E_{\xi_k}[\|\delta_k\|^2|\xi_{[k-1]}]\le\sigma^2/m_k$. Therefore, taking the expectation on both sides of \eqref{sum-gra} with respect to $\xi_{[N]}$ yields
\be\label{int-exp}
\sum_{k=1}^N
 ( m\alpha_k - L M^2\alpha_k^2/2)\E_{\xi_{[N]}}[\|\nabla f(x_k)\|^2] \le f(x_1) - f^{low} + \frac{LM^2\sigma^2}{2}\sum_{k=1}^N\frac{\alpha_k^2}{m_k}.
\ee
Since $R$ is a random variable with probability mass function $P_R$ given in \eqref{P_R}, it follows that
\be
\E[\|\nabla f(x_R)\|^2] = \E_{R,\xi_{[N]}}[\|\nabla f(x_R)\|^2] = \frac{\sum_{k=1}^N
\left ( m\alpha_k - L M^2\alpha_k^2/2\right)\E_{\xi_{[N]}}[\|\nabla f(x_k)\|^2]}{\sum_{k=1}^N
\left ( m\alpha_k - L M^2\alpha_k^2/2\right)},
\ee
which together with \eqref{int-exp} implies \eqref{exp-gra}.
\end{proof}

\begin{rem}
Different from SQN (Algorithm \ref{sso-uncons}), stepsizes $\alpha_k$ in RSQN (Algorithm \ref{finit-sso-uncons}) are not required to satisfy the condition \eqref{alpha-inf}. Besides, the assumption on the boundedness of $\{\|\nabla f(x_k)\|\}$ is not needed in RSQN.
\end{rem}

The following complexity result follows immediately from Theorem \ref{thm3.2}.
\begin{cor}\label{cor2.2}
Under the same conditions as in Theorem \ref{thm3.2}, further assume that the stepsizes $\alpha_k = m/(LM^2)$ and the batch sizes $m_k=\bar{m}$ for all $k=1,\ldots,N$ for some integer $\bar{m}\ge 1$. Then the following holds
\be \label{exp-g}
\E[\|\nabla f(x_R)\|^2] \le \frac{2LM^2D_f}{Nm^2} + \frac{\sigma^2}{\bar{m}},
\ee
where the expectation is taken with respect to $R$ and $\xi_{[N]}$.
\end{cor}

From Corollary \ref{cor2.2} we can see that the right hand side of \eqref{exp-g} depends on the batch size $\bar{m}$. Once $\bar{m}$ is fixed, no matter how large the maximum iteration number $N$ is, the right hand side of \eqref{exp-g} is always lower bounded by $\sigma^2/\bar{m}$. Since we want $\E[\|\nabla f(x_R)\|^2]$ to be as small as possible, we expect that it approaches zero when $N$ is sufficiently large. Therefore, $\bar{m}$ has to be chosen properly. The following corollary provides a choice of $\bar{m}$ such that the worst-case $\SFO$-calls complexity of RSQN method is in the order of $O(\epsilon^{-2})$ for obtaining an $\epsilon$-solution.

\begin{cor}\label{cor2.3}
Let $\bar{N}$ be the total number of $\SFO$-calls needed to calculate stochastic gradient $G_k$ in Step 2 of Algorithm \ref{finit-sso-uncons} for all the iterations. Under the same conditions as in Corollary \ref{cor2.2}, if we further assume that the batch size $m_k$ is defined as
\be \label{batch-size-m}
m_k = \bar{m} := \left\lceil \min \left\{ \bar{N}, \max\left\{1, \frac{\sigma}{L}\sqrt{\frac{\bar{N}}{\tilde{D}}}\right\}\right\}\right\rceil,
\ee
where $\tilde{D}$ is some problem-independent positive constant, then we have
\be
\E[\|\nabla f(x_R)\|^2]   \le   \frac{4LM^2 D_{f} }{\bar{N}m^2}\left(1+\frac{\sigma}{L}\sqrt{\frac{\bar{N}}{\tilde{D}}}\right) +  \max\left\{\frac{ \sigma^2}{\bar{N}}, \frac{ \sigma L\sqrt{\tilde{D}}}{\sqrt{\bar{N}}}\right\},
\ee
where the expectation is taken with respect to $R$ and $\xi_{[N]}$.
\end{cor}

\begin{proof}
Note that the number of iterations of Algorithm \ref{finit-sso-uncons} is at most $N=\lceil \bar{N}/\bar{m}\rceil$. Obviously, $N\ge \bar{N}/(2\bar{m})$. From Corollary \ref{cor2.2} we have that
\begin{align}
\E[\|\nabla f(x_R)\|^2] & \le \frac{2L M^2D_{f}}{Nm^2} + \frac{\sigma^2}{\bar{m}}  \le  \frac{4LM^2 D_{f}}{\bar{N}m^2}\bar{m} + \frac{\sigma^2}{\bar{m}} \label{exp-1} \\
& \le  \frac{4LM^2 D_{f} }{\bar{N}m^2}\left(1+\frac{\sigma}{L}\sqrt{\frac{\bar{N}}{\tilde{D}}}\right) +  \max\left\{\frac{ \sigma^2}{\bar{N}}, \frac{ \sigma L\sqrt{\tilde{D}}}{\sqrt{\bar{N}}}\right\},  \notag
\end{align}
which completes the proof.
\end{proof}

The following corollary follows immediately from Corollary \ref{cor2.3}.
\begin{cor}\label{cor3.4}
Under the same conditions as Corollary \ref{cor2.3}, for any given $\epsilon>0$, we further assume that the total number of $\SFO$ calls $\bar{N}$ to calculate $G_k$ in Step 2 of Algorithm \ref{finit-sso-uncons} satisfies
\be \label{bar-N}
\bar{N}\ge\max\left\{\frac{C_1^2}{\epsilon^2} + \frac{4C_2}{\epsilon},\frac{\sigma^2}{L^2\tilde{D}}\right\}
\ee
where
\[
C_1=\frac{4\sigma M^2D_f}{m^2\sqrt{\tilde{D}}} + \sigma L\sqrt{\tilde{D}}, \quad C_2=\frac{4LM^2D_f}{m^2},
\]
and $\tilde{D}$ is same as in \eqref{batch-size-m}. Then we have
\[
\E[\|\nabla f(x_R)\|^2]\le\epsilon,
\]
where the expectation is taken with respect to $R$ and $\xi_{[N]}$. It follows that to achieve $\E[\|\nabla f(x_R)\|^2]\le \epsilon$, the number of $\SFO$-calls needed to compute $G_k$ in Step 2 of Algorithm \ref{finit-sso-uncons} is at most in the order of $O(\epsilon^{-2})$.
\end{cor}

\begin{proof}
\eqref{bar-N} indicates that
\[ \sqrt{\bar{N}} \ge\frac{\sqrt{C_1^2 + 4\epsilon C_2}}{\epsilon } \ge \frac{\sqrt{C_1^2 + 4\epsilon C_2} + C_1}{2\epsilon}. \]
\eqref{bar-N} also implies that $\sigma^2/\bar{N}\le \sigma L\sqrt{\tilde{D}}/\sqrt{\bar{N}}$. Then from Corollary \ref{cor2.3} we have that
\[
\E[\|\nabla f(x_R)\|^2]  \le \frac{4LM^2 D_{f} }{\bar{N}m^2}\left(1+\frac{\sigma}{L}\sqrt{\frac{\bar{N}}{\tilde{D}}}\right) +  \frac{ \sigma L\sqrt{\tilde{D}}}{\sqrt{\bar{N}}}
 = \frac{C_1}{\sqrt{\bar{N}}} + \frac{C_2}{\bar{N}}
 \le \epsilon.
\]
\end{proof}

\begin{rem}\label{rem3.2}
In Corollaries \ref{cor2.3} and \ref{cor3.4} we did not consider the $\SFO$-calls that may be involved in updating $B_{k+1}$ in Step 4 of the algorithms. In the next section, we will consider two specific updating schemes for $B_k$, and analyze their $\SFO$-calls complexities for calculating $B_k$.
\end{rem}

\section{Two specific updating schemes for $B_k$}\label{sec:SDBFGS}

In Sections \ref{sec:RSSA} and \ref{sec:rssa}, we proposed two general frameworks for stochastic quasi-Newton methods for solving \eqref{orig-prob} and analyzed their convergence and worst-case $\SFO$-calls complexity, respectively. In both frameworks, we require that the Hessian approximation $B_k$ satisfies assumptions {\bf AS.3} and {\bf AS.4}. In this section, we study two specific updating schemes for $B_k$ such that {\bf AS.3} and {\bf AS.4} always hold.

\subsection{Stochastic damped BFGS updating formula}

In the setting of deterministic optimization, the classical BFGS algorithm updates the $B_k$ through the formula \eqref{bfgs}. It can be proved that $B_{k+1}$ is positive definite as long as $B_k$ is positive definite and $s_k^\top y_k>0$ (see, e.g., \cite{nw06,sy06}). Line search techniques are usually used to ensure that $s_k^\top y_k>0$ is satisfied. However, in stochastic quasi-Newton method, line search techniques cannot be used because the objective function value is assumed to be difficult to obtain. As a result, how to preserve the positive definiteness of $B_k$ is a main issue in designing stochastic quasi-Newton algorithms.

In \cite{mr10}, the RES algorithm is proposed for strongly convex stochastic optimization, in which iterates are updated via \eqref{x-k} where $\zeta_k$ is set as a positive constant $\Gamma$.
The following formula is adopted in \cite{mr10} for calculating the difference of the gradients:
\[
\hat{y}_k = \bar{G}_{k+1} - G_k - \hat{\delta} s_k,
\]
where $\hat{\delta}>0$ and
\[
\bar{G}_{k+1} := \frac{1}{m_k}\sum_{i=1}^{m_k}G(x_{k+1},\xi_{k,i}).
\]
It should be noted that the same sample set $\{\xi_{k,1},\ldots,\xi_{k,m_k}\}$ is used to compute $G_k$ and $\bar{G}_{k+1}$. $B_{k+1}$ is then calculated by the shifted BFGS update:
\be \label{re-bfgs}
B_{k+1} = B_k + \frac{\hat{y}_k\hat{y}_k\tr }{s_k\tr \hat{y}_k} - \frac{B_ks_ks_k\tr B_k}{s_k\tr B_ks_k} + \hat{\delta} I,
\ee
where the shifting term $\hat{\delta}I$ is added to prevent $B_{k+1}$ from being close to singular.
It is proved in \cite{mr10} that $B_{k+1}\succeq\hat{\delta}I$ under the assumption that $f$ is strongly convex. However, \eqref{re-bfgs} cannot guarantee the positive definiteness of $B_{k+1}$ for nonconvex problems. Hence, we propose the following stochastic damped BFGS updating procedure (Procedure \ref{DBFGS}) for nonconvex problems. The damped BFGS updating procedure has been used in sequential quadratic programming method for constrained optimization in deterministic setting (see, e.g., \cite{nw06}).

{\floatname{algorithm}{Procedure}
\begin{algorithm}[ht]
\caption{{\bf Stochastic Damped-BFGS update (SDBFGS) }}\label{DBFGS}
\begin{algorithmic}[1]
\REQUIRE {Given $\delta>0$, $\xi_k$, $B_k$, $G_k$, $x_k$ and $x_{k+1}$.}
\ENSURE {$B_{k+1}$.}
   \STATE Calculate $s_k=x_{k+1}-x_k$ and calculate $\hat{y}_k$ through
   \[
   \hat{y}_k = \bar{G}_{k+1} - G_k - \delta s_k,
   \]
   where $\bar{G}_{k+1} := \frac{1}{m_k}\sum_{i=1}^{m_k}G(x_{k+1},\xi_{k,i})$.
   \STATE Calculate
   \[\hat{r}_k=\hat{\theta}_k \hat{y}_k +(1-\hat{\theta}_k)B_ks_k,\]
    where $\hat{\theta}_k$ is calculated through:
\[
\hat{\theta}_k = \begin{cases}
1, & \mbox{if } s_k\tr \hat{y}_k \ge 0.2s_k\tr B_ks_k,\\
(0.8s_k\tr B_ks_k)/(s_k\tr B_ks_k - s_k\tr \hat{y}_k), & \mbox{if } s_k\tr \hat{y}_k<0.2s_k\tr B_ks_k.
\end{cases}
\]
    \STATE Calculate $B_{k+1}$ through
\be \label{sto-bfgs}
B_{k+1} = B_k + \frac{\hat{r}_k \hat{r}_k\tr }{s_k\tr \hat{r}_k} - \frac{B_ks_ks_k\tr B_k}{s_k\tr B_ks_k} + \delta I.
\ee
\end{algorithmic}
\end{algorithm}}

\begin{rem}
Notice that the most significant difference between Procedure \ref{DBFGS} and RES lies in that $\hat{r}_k$, which is a convex combination of $\hat{y}_k$ and $B_ks_k$, is used to replace $\hat{y}_k$ in the updating formula \eqref{sto-bfgs} for $B_{k+1}$.
\end{rem}

The following lemma shows that $\{B_k\}$ obtained by Procedure \ref{DBFGS} is uniformly positive definite.

\begin{lem}\label{lem5.1}
Suppose that $B_k$ is positive definite, then $B_{k+1}$ generated by Procedure \ref{DBFGS} satisfies
\be\label{B-k+1}
B_{k+1}\succeq\delta I.
\ee
\end{lem}
\begin{proof}
From the definition of $\hat{r}_k$, we have that
\[
s_k\tr \hat{r}_k = \hat{\theta}_k(s_k\tr \hat{y}_k - s_k\tr B_ks_k) + s_k\tr B_ks_k
               = \begin{cases}
               s_k\tr \hat{y}_k,\quad & \mbox{if }s_k\tr \hat{y}_k \ge 0.2 s_k\tr B_ks_k,\\
               0.2s_k\tr B_ks_k,\quad & \mbox{if }s_k\tr \hat{y}_k < 0.2 s_k\tr B_ks_k,
               \end{cases}
\]
which implies $s_k\tr\hat{r}_k\geq 0.2s_k\tr B_ks_k$.
Denote $u_k=B_k^{\frac{1}{2}}s_k$. Then we have
\[
B_k - \frac{B_ks_ks_k\tr B_k}{s_k\tr B_k s_k} = B_k^{\frac{1}{2}}\left(I - \frac{u_ku_k\tr}{u_k\tr u_k}\right)B_k^{\frac{1}{2}}.
\]
Since $I - \frac{u_ku_k\tr}{u_k\tr u_k}\succeq 0$ and $s_k\tr \hat{r}_k>0$, we have that $B_k + \frac{\hat{r}_k \hat{r}_k\tr }{s_k\tr \hat{r}_k} - \frac{B_ks_ks_k\tr B_k}{s_k\tr B_ks_k}\succeq 0$. It then follows from \eqref{sto-bfgs} that $B_{k+1}\succeq \delta I$.
\end{proof}

From Lemma \ref{lem5.1} we can see that, if starting with $B_1\succeq\delta I$, we have $B_k\succeq \delta I$ for all $k$. So if we further choose $\zeta_k\ge\zeta$ for any positive constant $\zeta$, then it holds that
\[
\zeta I\preceq B_k^{-1}+\zeta_k I\preceq \left(\frac{1}{\delta}+\zeta\right)I, \quad \mbox{for all $k$},
\]
which satisfies the assumption {\bf AS.3} with $m=\zeta$ and $M=\zeta+1/\delta$. Moreover, Since $\bar{G}_{k+1}$ is dependent only on $\xi_k$, it follows from \eqref{sto-bfgs} that $B_{k+1}$ is dependent only on $\xi_{[k]}$, which satisfies the assumption {\bf AS.4}.
Therefore, we conclude that assumptions {\bf AS.3} and {\bf AS.4} hold for $B_k$ generated by Procedure \ref{DBFGS}.
We should also point out that in stochastic damped BFGS update Procedure \ref{DBFGS}, the shifting parameter $\delta$ can be any positive scalar. But the shifting parameter $\hat{\delta}$ in \eqref{re-bfgs} used in RES is required to be smaller than the smallest eigenvalue of the  Hessian of the strongly convex function $f$, which is usually negative for nonconvex problem.

Note that in Step 1 of Procedure \ref{DBFGS}, the stochastic gradient at $x_{k+1}$ that is dependent on $\xi_{k}$ is computed. Thus, when Procedure \ref{DBFGS} is called at the $k$-th iteration to generate $B_{k+1}$ in Step 4 of Algorithm \ref{finit-sso-uncons}, another $m_k$ $\SFO$-calls are needed. As a result, the number of $\SFO$-calls at the $k$-th iteration of Algorithm \ref{finit-sso-uncons} becomes $2m_k$. This leads to the following complexity result for Algorithm \ref{finit-sso-uncons}.

\begin{thm}
Denote $N_{sfo}$ as the total number of $\SFO$-calls in Algorithm \ref{finit-sso-uncons} with Procedure \ref{DBFGS} to generate $B_{k+1}$. Under the same conditions as in Corollary \ref{cor3.4}, to achieve  $\E[\|\nabla f(x_R)\|^2]\le \epsilon$, $N_{sfo}$ is at most $2\bar{N}$ where $\bar{N}$ satisfies \eqref{bar-N}, i.e., is in the order of $O(\epsilon^{-2})$.
\end{thm}

\subsection{Stochastic cyclic-BB-like updating formula}\label{sec:SBB}

Note that computing $B_k^{-1}G_k$ in the updating formula for $x_k$ \eqref{x-k} might be costly if $B_k$ is dense or the problem dimension is large. To overcome this potential difficulty, we propose a cyclic Barzilai-Borwein (BB) like updating formula for $B_k$ in this section. This updating formula can ensure that $B_k$ is a diagonal matrix and thus very easy to be inverted.

The BB method has been studied extensively since it was firstly proposed in \cite{BB88}. BB method is a gradient method with certain properties of quasi-Newton method. At the $k$-th iteration, the step size $\alpha_k^{\mbox{\tiny BB}}$ for the gradient method is calculated via
\[
\alpha_k^{\mbox{\tiny BB}} := \mbox{arg}\min_{\alpha\in\R} \|\alpha s_k - y_k\|^2, \quad \mbox{ or } \quad \alpha_k^{\mbox{\tiny BB}} := \mbox{arg}\min_{\alpha\in\R} \|s_k - y_k/\alpha\|^2,
\]
where $s_k:=x_{k}-x_{k-1}$, $y_k:=\nabla f(x_k) - \nabla f(x_{k-1})$.
Direct calculations yield
\[\alpha_k^{\mbox{\tiny BB}}=\frac{s_k\tr y_k}{\|s_k\|^2}, \quad\mbox{ or }\quad \alpha_k^{\mbox{\tiny BB}}=\frac{\|y_k\|^2}{s_k\tr y_k}.\]
Many studies have shown the superiority of BB methods over the classical gradient descent method in both theory and practical computation. Readers are referred to \cite{HagerMairZhang09} for a relatively comprehensive discussion on BB methods. Besides, BB methods have been applied to solve many problems arising in real applications, such as image reconstruction \cite{wang-ma-bb-2007,NYFPZ14} and electronic structure calculation \cite{ZZWZ14}, and they have shown promising performance.
Recently, the nice numerical behavior of cyclic BB (CBB) methods attracts a lot of attentions (see, e.g., \cite{DHSZ06,HagerMairZhang09}). In CBB method, BB stepsize is used cyclicly, i.e., the stepsize in the $l$-th cycle is
\[
\alpha_{ql+i} = \alpha_{ql+1}^{\mbox{\tiny BB}},\quad i = 1,\ldots,q,
\]
where $q\ge1$ is the cycle length and $l=0,1,\ldots$.
In the setting of deterministic optimization, line search techniques are usually adopted in CBB to ensure the global convergence.
Although line search techniques are not applicable in stochastic optimization, we can still apply the idea of CBB to design an efficient algorithm that does not need to compute matrix inversion or solve linear equations in \eqref{x-k}. The details of our procedure to generate $B_k$ using stochastic CBB-like method are described as follows.

We set $B_k := \lambda_k^{-1} I$, and $\lambda_k$ is updated as in CBB method $\lambda_{ql+i} = \lambda_{ql+1}^{\tiny BB},i = 1,\ldots,q$, where $q$ is the cycle length and $l=0,1,\ldots$, and $\lambda_{ql+1}^{\tiny BB}$ is the optimal solution to
\be\label{ls}
\min_{\lambda\in\R}\quad  \|\lambda^{-1} s_{ql}-y_{ql}\|^2, \quad \mbox{ or } \quad \min_{\lambda\in\R}\quad  \| s_{ql}- \lambda y_{ql}\|^2,
\ee
where $s_k=x_{k+1}-x_{k}$ and the gradient difference $y_k$ is defined as
\be \label{y}
  y_k = \bar{G}_{k+1} - G_k = \frac{\sum_{i=1}^{m_{k}}G(x_{k+1},\xi_{k,i})}{m_{k}} - \frac{\sum_{i=1}^{m_{k}}G(x_{k},\xi_{k,i})}{m_{k}}.
\ee
Direct calculations yield that $\lambda_{ql+1}^{\tiny BB}=s_{ql}\tr y_{ql}/\|y_{ql}\|^2$ or $\lambda_{ql+1}^{\tiny BB}=\|s_{ql}\|^2/s_{ql}\tr y_{ql}$. However, $\lambda_{ql+1}^{\tiny BB}$ calculated in this way might be negative since $s_{ql}\tr y_{ql}<0$ might happen. Therefore, we must adapt the stepsize in order to preserve the positive definiteness of $B_k$. We thus propose the following strategy for calculating $\lambda_k$:
\be\label{lamb}
\lambda_{k+1} =
\begin{cases}
\lambda_k,\quad & \mbox{if $mod(k,q) \neq 0$,}\\
1, \quad & \mbox{if $mod(k,q) = 0$, $s_k\tr y_k\le 0$}\\
\mathrm{P}_{[\lambda_{\min},\lambda_{\max}]} \frac{s_{ql}\tr y_{ql}}{\|y_{ql}\|^2}\quad \mbox{or}\quad \mathrm{P}_{[\lambda_{\min},\lambda_{\max}]} \frac{\|s_{ql}\|^2}{s_{ql}\tr y_{ql}}, \quad & \mbox{if $mod(k,q) = 0$, $s_k\tr y_k > 0$},
\end{cases}
\ee
where $\mathrm{P}_{[\lambda_{\min},\lambda_{\max}]}$ denotes the projection onto the interval $[\lambda_{\min},\lambda_{\max}]$, where $\lambda_{\min}$ and $\lambda_{\max}$ are given parameters. Note that we actually switch to gradient descent method (by setting $\lambda_k=1$) if $s_k\tr y_k<0$. In our numerical tests later we will report the frequency of BB steps in this procedure. Notice that $B_k$ generated in this way satisfies the assumption {\bf AS.3} with
\[
m=\min\{\lambda_{\min},1\},\quad M=\max\{\lambda_{\max},1\},
\]
and in this case we can set $\zeta_{k}=0$ for all $k$ in \eqref{x-k}.

The stochastic CBB updating procedure for $B_{k+1}$ is summarized formally in Procedure \ref{SCBB}.

{\floatname{algorithm}{Procedure}
\begin{algorithm}[ht]
\caption{{\bf Stochastic Cyclic-BB-like update (SCBB) }}\label{SCBB}
\begin{algorithmic}[1]
\REQUIRE {Given $q\in\mathbb{N}_+$, $G_k,\lambda_{\min},\lambda_{\max}\in\R^n$ with $0<\lambda_{\min} < \lambda_{\max}$, $\xi_k$, $G_k$, $x_k$ and $x_{k+1}$.}
\ENSURE {$B_{k+1}$.}
  \IF {$mod(k,q)=0$}
  \STATE Calculate $s_k = x_{k+1}-x_k$ and
  \[y_k = \frac{\sum_{i=1}^{m_{k}}G(x_{k+1},\xi_{k,i})}{m_{k}} - G_k;\]
      \IF {$s_k\tr y_k > 0$}
      \STATE $\lambda_{k+1} = \mathrm{P}_{[\lambda_{\min},\lambda_{\max}]}\frac{s_k\tr y_k}{\|y_k\|^2}$ or $\mathrm{P}_{[\lambda_{\min},\lambda_{\max}]}\frac{\|s_k\|^2}{s_k\tr y_k}$;
      \ELSE
      \STATE $\lambda_{k+1} = 1$;
      \ENDIF
  \ELSE
  \STATE $\lambda_{k+1} = \lambda_k$;
  \ENDIF
  \STATE Set $B_{k+1} = \lambda_{k+1}^{-1}I$.
\end{algorithmic}
\end{algorithm}}

When Procedure \ref{SCBB} is used to generate $B_{k+1}$ in Step 4 of Algorithm \ref{finit-sso-uncons}, we have the following complexity result on $\SFO$-calls.

\begin{thm}
Denote $N_{sfo}$ as the total number of $\SFO$-calls in Algorithm \ref{finit-sso-uncons} with Procedure \ref{SCBB} called to generate $B_{k+1}$ at each iteration. Under the same conditions as Corollary \ref{cor3.4}, to achieve $\E[\|\nabla f(x_R)\|^2]\le \epsilon$, $N_{sfo}$ is at most $\lceil (1+q)\bar{N}/q \rceil$ where $\bar{N}$ satisfies \eqref{bar-N}, i.e., $N_{sfo}$ is in the order of $O(\epsilon^{-2})$.
\end{thm}

\begin{proof}
Under the same conditions as Corollary \ref{cor3.4}, the batch size $m_k=\bar{m}$ for any $k$. If Procedure \ref{SCBB} is called at each iteration of Algorithm \ref{finit-sso-uncons}, then in every $q$ iterations, $\bar{m}(q+1)$ $\SFO$-calls are needed. Since to achieve $\E[\|\nabla f(x_R)\|^2]\le \epsilon$ the number $\SFO$ calls in Step 2 of Algorithm \ref{finit-sso-uncons} is at most $\bar{N}$, the total number of $\SFO$-calls in Algorithm \ref{finit-sso-uncons} is at most $\lceil (1+q)\bar{N}/q \rceil$.
\end{proof}

\section{Numerical Experiments}\label{sec:num}

In this section, we conduct numerical experiments to test the practical performance of the proposed algorithms.

By combining Algorithms \ref{sso-uncons} and \ref{finit-sso-uncons} with Procedures \ref{DBFGS} and \ref{SCBB}, we get the following four algorithms: SDBFGS (Algorithm \ref{sso-uncons} with Procedure \ref{DBFGS}), SCBB (Algorithm \ref{sso-uncons} with Procedure \ref{SCBB}), RSDBFGS (Algorithm \ref{finit-sso-uncons} with Procedure \ref{DBFGS}), and RSCBB (Algorithm \ref{finit-sso-uncons} with Procedure \ref{SCBB}). We compare them with three existing methods for solving \eqref{orig-prob}: SGD, RSG \cite{gl13} and RES \cite{mr10}.

Since the course of these algorithms is a stochastic process, we run each instance $N_{run}$ times and report the performance in average. In particular, we report the number of $\SFO$-calls ($N_{sfo}$), the CPU time (in seconds), and the mean and variance (var.) of {$\|\nabla f(x_k^*)\|$ (or $\|\nabla f(x_k^*)\|^2$) over $N_{run}$ runs, where $x_k^*$ is the output of the tested algorithm at $k$-th run with $k=1,\ldots,N_{run}$.

All the algorithms are implemented in Matlab R2013a on a PC with a 2.60 GHz Intel microprocessor and 8GB of memory.

\subsection{A convex stochastic optimization problem}\label{sec:num-strong}

We first consider a convex stochastic optimization problem, which is also considered in \cite{mr10}:
\be \label{prob_1}
\min_{x\in \R^n}\quad f(x)=\E_{\xi}[f(x,\xi)]:= \E [\frac{1}{2}x\tr (A+A\mbox{diag}(\xi))x - b\tr x ],
\ee
where $\xi$ is uniformly drawn from $\Xi:=[-0.1,0.1]^n$, $b$ is chosen uniformly randomly from $[0,1]^n$, and $A$ is a diagonal matrix whose diagonal elements are uniformly chosen from a discrete set $\C{S}$ which will be specified later. We can control the condition number of \eqref{prob_1} through the choice of $\C{S}$ and we will explore the performances of algorithms under different condition numbers.

For \eqref{prob_1}, we compare SDBFGS and SCBB with SGD and RES. For SGD, we tested two different choices of stepsize, i.e., $\alpha_k=10^2/(10^3+k)$ and $10^4/(10^4+k)$. We also tested some other choices for $\alpha_k$, but the performance with these two are  relatively better.
The parameters for the other three algorithms are set as follows:
\begin{equation*}
\begin{aligned}
\mbox{SCBB:}\quad &\alpha_k = \frac{10^2}{10^3 + k},\quad \lambda_{\min} = 10^{-6},\quad \lambda_{\max} = 10^8, \quad q=5, \\
\mbox{SDBFGS:}\quad & \alpha_k =\frac{10^2}{10^3+k},\quad \zeta_k=10^{-4},\quad \delta=10^{-3}, \\
\mbox{RES:} \quad & \alpha_k=\frac{10^2}{10^3+k}, \quad \Gamma=10^{-4}, \quad \hat{\delta}=10^{-3}.
\end{aligned}
\end{equation*}
Note that the parameter settings for RES are the same as the ones used in \cite{mr10}. To make a fair comparison with RES, we thus adopted the same stepsize in these three algorithms above.

Since the solution of \eqref{prob_1}
is $x^* = A^{-1}b$ if the random perturbation is ignored, we terminate the algorithms when
\[
\frac{\|x_k-x^*\|}{\max\{1,\|x^*\|\}}\le \rho,
\]
where $\rho>0$ is a given tolerance. We chose $\rho=0.01$ in our experiments. We set the batch size $m_k=5$ for all the tested algorithms. Besides, for each instance the maximum iteration number is set as $10^4$. The results for different dimension $n$ and set $\C{S}$ are reported in Table \ref{table1}. Note that different choices of $\C{S}$ can reflect different condition numbers of \eqref{prob_1}.

\begin{table}
\caption{Results for solving \eqref{prob_1}. Mean value and variance (var.) of $\{\|\nabla f(x_k^*)\|:k=1,\ldots,N_{run}\}$ with $N_{run}=20$ are reported. ``---'' means that the algorithm is divergent.}
\label{table1}
\begin{center} \footnotesize
\begin{tabular}{|c|c|c|c|c|c|c|}
  \hline

 $n$ &   & SGD & SGD &  RES & SDBFGS& SCBB  \\
     &   &   $\alpha_k=\frac{10^2}{10^3+k} $  & $\alpha_k=\frac{10^4}{10^4+k} $ &  $\alpha_k=\frac{10^2}{10^3+k} $ & $\alpha_k=\frac{10^2}{10^3+k} $ & $\alpha_k=\frac{10^2}{10^3+k} $   \\
 \hline
  &  & \multicolumn{5}{c|}{$\C{S}=\{0.1,1\}$}  \\
  \hline
 \multirow{4}*{500}& $N_{sfo}$ &2.921e+03 & 2.400e+02 & 5.035e+02 &  5.025e+02  & 7.653e+02\\
& mean & 9.781e-02 & 2.446e-01 & 9.933e-02 & 1.002e-01  & 1.123e-01 \\
& var. & 7.046e-07 & 1.639e-04 & 4.267e-06 & 7.329e-06 &  6.020e-06 \\
& CPU & 2.848e-01 & 2.580e-02 & 9.607e-01 &  6.273e-01 &  3.095e-02  \\
  \hline
  \multirow{4}*{1000} & $N_{sfo}$ & 2.925e+03 & 2.380e+02 & 5.015e+02  & 5.000e+02 &   7.243e+02 \\
& mean & 1.453e-01 & 3.532e-01 & 1.476e-01  & 1.474e-01  & 1.667e-01 \\
& var. & 1.101e-06 & 1.368e-04 & 1.117e-05 &  8.251e-06 &  9.984e-06  \\
& CPU & 1.172e+00 & 8.665e-02 & 6.031e+00 &  6.109e+00  & 2.924e-01  \\
\hline
  \multirow{4}*{5000} & $N_{sfo}$ & 2.924e+03 & 2.400e+02 & 5.045e+02 & 5.045e+02  & 7.575e+02\\
& mean & 3.165e-01 & 7.982e-01 & 3.194e-01 &  3.180e-01 &  3.624e-01  \\
& var. & 1.255e-06 & 2.050e-04 & 1.336e-05 & 1.246e-05 & 8.917e-06  \\
& CPU & 1.270e+01 & 7.661e-01 & 3.492e+02 & 3.577e+02 & 2.371e+00  \\
  \hline
& & \multicolumn{5}{c|}{ $\C{S}=\{0.1,1,10\}$ } \\
\hline
\multirow{4}*{500} & $N_{sfo}$ & 2.927e+03 & 5.000e+04 & 2.865e+02 &   2.875e+02  &  8.315e+03  \\
& mean & 1.622e-01 & --- & 6.016e-01 &   5.698e-01 &   9.429e-02  \\
& var. & 5.421e-05 & --- & 1.162e-02 &   9.589e-03 &   3.281e-06  \\
& CPU & 3.041e-01 & 4.842e+00 & 5.132e-01 &   3.683e-01 &   7.994e-01  \\
\hline
\multirow{4}*{1000} & $N_{sfo}$ & 2.928e+03 & 5.000e+04 & 2.875e+02 &   2.880e+02 &   7.101e+03  \\
& mean & 2.137e-01 & --- & 7.707e-01 &   7.791e-01 &   1.372e-01  \\
& var. & 4.638e-05 & --- & 7.222e-03 &   6.860e-03 &   2.834e-06  \\
& CPU & 9.459e-01 & 1.934e+01 & 3.354e+00 &   3.595e+00 &   2.855e+00  \\
\hline
\multirow{4}*{5000} & $N_{sfo}$ & 2.925e+03 & 5.000e+04 & 2.865e+02 &   2.865e+02 &   8.035e+03  \\
& mean & 4.911e-01 & --- & 1.957e+00 &   1.956e+00 &   2.903e-01  \\
& var. & 6.575e-05 & --- & 3.916e-02 &   4.517e-02 &   5.618e-06  \\
& CPU & 1.564e+01 & 1.525e+02 & 2.023e+02 &   2.039e+02 &   2.254e+01  \\
\hline
&&\multicolumn{5}{c|}{ $\C{S}=\{0.1,1,10,100\}$  } \\
\hline%
\multirow{4}*{500} & $N_{sfo}$ & 5.000e+04 & 5.000e+04 & 6.279e+03 &  6.409e+03 & 4.953e+04  \\
& mean & --- & --- & 3.193e-01 &  3.479e-01 &  2.049e-01  \\
& var. & --- & --- & 6.003e-02 &   6.754e-02 &   9.889e-05  \\
& CPU & 3.136e+00 & 3.517e+00 &  8.458e+00   & 9.817e+00 &   3.955e+00  \\
\hline
\multirow{4}*{1000} & $N_{sfo}$ & 5.000e+04 & 5.000e+04 & 9.028e+03  &  9.016e+03 &   5.644e+04   \\
& mean & --- & --- & 5.615e-01 &   5.005e-01 &   2.397e-01  \\
& var. & --- & --- & 6.704e-02 &   8.857e-02 &   2.132e-04  \\
& CPU &1.774e+01 & 1.203e+01 &  1.251e+02 &   1.267e+02 &   1.169e+01  \\
\hline
\multirow{4}*{5000} & $N_{sfo}$ & 5.000e+04 & 5.000e+04 & 6.756e+03 &   6.694e+03 &   6.000e+04  \\
& mean & --- & --- & 9.388e+00 &   1.104e+01 &   1.118e+00  \\
& var. & --- & --- & 4.133e+01 &   5.345e+01 &   2.581e-04  \\
& CPU &1.534e+02 & 3.041e+03 & 3.022e+03 &   5.820e+03 &   1.278e+02  \\
\hline
\end{tabular}
\end{center}
\end{table}

From Table \ref{table1} we see that the performance of SGD is poor compared with the other three methods. The average number of $\SFO$-calls of SGD is significantly larger than the ones given by RES, SDBFGS and SCBB. Moreover, SGD diverges if the stepsize $\alpha_k$ is too large or the condition number of the problem increases. It is also noticed that the performance of RES and SDBFGS is comparable. Furthermore, SCBB seems to be the best among the tested algorithms in terms of mean and variance of $\|\nabla f(x_k^*)\|$ as well as the CPU time, although RES and SDBFGS need less number of $\SFO$-calls.

\subsection{A nonconvex support vector machine problem}

In this section, we compare RSDBFGS and RSCBB with RSG studied in \cite{gl13} for solving the following nonconvex support vector machine problem with a sigmoid loss function (see \cite{Mason-nips-1999})
\be\label{prob_SVM}
\min_{x\in\R^n}\quad f(x):=\E_{u,v}[1-\mathrm{tanh}(v\langle x,u\rangle)] + \lambda \|x\|_2^2
\ee
where $\lambda>0$ is a regularization parameter, $u\in\R^n$ denotes the feature vector, $v\in\{-1,1\}$ refers to the corresponding label and $(u,v)$ is drawn from the uniform distribution on $[0,1]^n\times \{-1,1\}$. Note that we do not compare with RES here because RES is designed for solving strongly convex problems. In order to compare with RSG, we adopt the same experimental settings as in \cite{gl13-ex}. The regularization parameter $\lambda$ is set as 0.01. The initial point is set as $x_1=5*\bar{x}_1$, where $\bar{x}_1$ is drawn from the uniform distribution over $[0,1]^n$. At the $k$-th iteration, to compute the stochastic gradient at iterate $x_k$, a sparse vector $u_k$ with 5\% nonzero components is first generated following the uniform distribution on $[0,1]^n$, and then $v_k$ is computed through $v_k=\mbox{sign}(\langle \bar{x},u_k \rangle)$ for some $\bar{x}\in\R^n$ drawn from uniform distribution on $[-1,1]^n$. Note that here the batch size $m_k$ is equal to 1.

The code of RSG was downloaded from http://www.ise.ufl.edu/glan/computer-codes. In order to make a fair comparison, we generate our codes RSDBFGS and RSCBB by replacing the update formula \eqref{sgd} in RSG by SDBFGS and SCBB procedures. In both RSDBFGS and RSCBB, we adopt the same stepsize as in RSG. Note that an auxiliary routine is implemented to estimate the Lipshitz constant in \cite{gl13-ex}. The cycle length in SCBB is set as $q=5$. We test the three algorithms with different problem sizes $n=500$, $1000$ and $2000$ and different number of $\SFO$-calls $N_{sfo}=2500$, $5000$, $10000$ and $20000$. Recall that the theoretical performance of expectation of squared norm of gradient at returned point has been analyzed in Section \ref{sec:rssa}. We next report the mean value and variance of $\|\nabla f(x_k^*)\|^2$ over $N_{run}=20$ runs of each algorithm solving \eqref{prob_SVM} in Table \eqref{tab_SVM_1}. To evaluate the quality of $x_k^*$ in terms of classification, we also report the misclassification error on a testing set $\{(u_i,v_i):i=1,\ldots,K\}$, which is defined as
\[
\mbox{err}(x_k^*):= \frac{|\{i:v_i\neq \mbox{sign}(\langle x_k^*, u_i \rangle), i=1,\ldots, K\}|}{K},
\]
and the sample size $K=75000$. Here, the testing set is generated in the same way as we have introduced in previous paragraph.

\begin{table}
\caption{Results of RSG, RSDBFGS and RSCBB for solving \eqref{prob_SVM}. Mean value and variance (var.) of $\{\|\nabla f(x_k^*)\|^2: k=1,\ldots,N_{run}\}$ and average classification error (err.) with $N_{run}=20$ are reported.}\label{tab_SVM_1}
\begin{center} \footnotesize
\begin{tabular}{|c|c|c|c|c|}
  \hline


 $N_{sfo}$ &   & RSG & RSDBFGS & RSCBB \\
 \hline
\multicolumn{5}{|c|}{$n=500$} \\
  \hline
 \multirow{4}*{2500} & mean & 3.622e-01& 1.510e-02  & 3.021e-02\\
                     & var. & 1.590e-01& 3.041e-05 & 4.006e-04\\
                     & err.(\%)  & 49.13 & 33.34 & 40.09\\
                     & CPU & 3.794e+00  & 1.799e+01 & 4.136e+00\\
 \hline
 \multirow{4}*{5000} & mean & 2.535e-01& 1.441e-02 & 2.146e-02\\
                     & var.  &1.320e-01 & 3.887e-05 & 3.464e-04\\
                     & err.(\%)  &45.47 & 31.09 & 36.37\\
                     & CPU & 5.100e+00  & 3.452e+01 &  5.433e+00\\
 \hline
 \multirow{4}*{10000} & mean &  3.201e-01 &  1.033e-02   &  6.277e-02 \\
                     & var.  &  3.052e-01 & 1.770e-05  & 1.692e-02 \\
                     & err.(\%)  &  38.12 & 25.19  & 35.99 \\
                     & CPU & 7.127e+00  & 6.759e+01 & 6.969e+00 \\
 \hline
 \multirow{4}*{20000} & mean & 2.733e-01 & 1.117e-02  &  4.710e-02  \\
                      & var.  & 2.923e-01 & 8.284e-05  & 6.327e-03 \\
                      & err.(\%)  & 30.81 & 24.59 &  31.60\\
                      & CPU & 1.179e+01 & 1.227e+02 & 1.201e+01 \\
 \hline
\multicolumn{5}{|c|}{$n=1000$} \\
\hline
 \multirow{4}*{2500} & mean & 8.055e-01 & 1.966e-02  & 2.632e-02\\
                     & var.  & 4.250e-01  & 2.824e-05 & 1.087e-03\\
                     & err.(\%)  & 48.39 & 37.07 & 45.32\\
                     & CPU & 6.810e+00  & 8.172e+01 & 6.607e+00\\
 \hline
 \multirow{4}*{5000} & mean & 5.188e-01 & 1.823e-02 & 3.088e-02\\
                     & var.  &3.595e-01 & 1.717e-05  &  2.115e-04\\
                     & err.(\%)  &47.73 & 33.87 & 41.86\\
                     & CPU & 8.059e+00  & 1.586e+02 & 8.296e+00\\
  \hline
 \multirow{4}*{10000} & mean & 5.819e-01 & 1.744e-02  &  4.359e-02  \\
                     & var.  & 8.432e-01 &  6.141e-05  & 1.543e-03\\
                     & err.(\%)  &  45.44 & 30.61  & 40.07\\
                     & CPU & 1.026e+01  &  3.890e+02 & 1.032e+01 \\
 \hline
  \multirow{4}*{20000} & mean & 4.457e-01 & 1.542e-02  & 5.062e-02 \\
                      & var.  & 6.893e-01 & 3.920e-05  & 3.031e-03 \\
                      & err.(\%)  & 37.42 & 27.43  & 37.50 \\
                      & CPU & 1.841e+01 & 5.717e+02 & 1.673e+01 \\
 \hline
\multicolumn{5}{|c|}{$n=2000$} \\
\hline
 \multirow{4}*{2500} & mean & 2.731e+00 & 2.456e-02 & 5.366e-02\\
                     & var.  & 3.654e+00 & 1.019e-04 & 6.669e-03\\
                     & err.(\%)  &48.80 & 42.64 & 46.00\\
                     & CPU &  1.219e+01  &4.025e+02 & 1.246e+01\\
 \hline
 \multirow{4}*{5000} & mean & 2.153e+00 & 2.076e-02 & 4.303e-02\\
                     & var.  & 3.799e+00 & 1.353e-04 & 1.211e-03\\
                     & err.(\%)  &48.54 & 40.97 & 44.43\\
                     & CPU & 1.332e+01 & 1.048e+03 &  1.312e+01\\
  \hline
 \multirow{4}*{10000} & mean & 8.216e-01 & 2.546e-02  & 3.756e-02  \\
                     & var.  & 7.965e-01 & 1.084e-04 & 7.044e-04\\
                     & err.(\%)  &  46.98 & 38.73  & 42.50 \\
                     & CPU & 1.840e+01  & 1.578e+03  & 1.716e+01 \\
 \hline
  \multirow{4}*{20000} & mean & 4.570e-01  & 2.199e-02 & 5.265e-02 \\
                      & var.  & 3.782e-01  &  1.954e-04 & 2.820e-03 \\
                      & err.(\%) & 44.00  & 33.81 & 40.94 \\
                      & CPU & 2.399e+01  & 3.215e+03 & 2.448e+01 \\
 \hline

\end{tabular}
\end{center}
\end{table}

From Table \ref{tab_SVM_1} we have the following observations. First, both RSDBFGS and RSCBB outperform RSG in terms of mean value and variance of $\|\nabla f(x_k^*)\|^2$, and in all cases RSDBFGS is the best. Second, both RSDBFGS and RSCBB outperform RSG in most cases in terms of misclassification error, and RSDBFGS is always the best. Third, RSDBFGS consumes most CPU time and RSG and RSCBB are comparable in terms of CPU time. Fourth, for fixed $n$, the misclassification error decreases when $N_{sfo}$ increases.

Finally, we conduct some further tests to study the behavior of RSCBB. Note that in Procedure \ref{SCBB} we need to switch to a gradient step whenever $s_k\tr y_k<0$ happens. So it is important to learn how often this happens in the course of the algorithm. In Table \ref{tab_per_BB_SVM_1} we report the percentage of BB steps in RSCBB for solving \eqref{prob_SVM}. We can see from Table \ref{tab_per_BB_SVM_1} that for fixed $n$, the percentage of BB steps monotonically decreases when $N_{sfo}$ increases. Nonetheless, as we observed from the previous numerical tests, using BB steps helps significantly in improving the efficiency and accuracy of the algorithm.

\begin{table}
\caption{Percentage of BB steps in RSCBB for solving \eqref{prob_SVM}}\label{tab_per_BB_SVM_1}
\begin{center} \footnotesize
\begin{tabular}{|c|c|c|c|c|c|c|c|c|c|c|c|c|}
  \hline
$n$ & \multicolumn{4}{c|}{$500$} & \multicolumn{4}{c}{$1000$} & \multicolumn{4}{|c|}{$2000$} \\
\hline
$N_{sfo}$ & 2500 & 5000 & 10000 & 20000 & 2500 & 5000 & 10000 & 20000 & 2500 & 5000 & 10000 & 20000 \\
\hline
Per.(\%) & 66.15 & 54.80 & 53.21 & 53.15 & 56.70 & 48.45 & 46.33 & 40.66 & 57.12 & 45.49 & 32.65 & 30.04 \\
\hline
\end{tabular}
\end{center}
\end{table}

\section{Conclusions and remarks}\label{sec:conclusions}

In this paper we proposed two classes of stochastic quasi-Newton methods for nonconvex stochastic optimization. We first proposed a general framework of stochastic quasi-Newton methods, and analyzed its theoretical convergence in expectation. We further proposed a general framework of randomized stochastic quasi-Newton methods and established its worst-case $\SFO$-calls complexity. This kind of methods do not require the stepsize to converge to zero and provide an explicit worst-case $\SFO$-calls complexity bound. To create positive definite Hessian approximations that satisfy the assumptions required in the convergence and complexity analysis, we proposed two specific stochastic quasi-Newton update strategies, namely SDBFGS and SCBB strategies. We also studied their worst-case $\SFO$-calls complexities in the corresponding stochastic quasi-Newton algorithms. Finally, we reported some numerical experimental results that demonstrate the efficiency of our proposed algorithms. The numerical results indicate that the proposed SDBFGS and SCBB are preferable compared with some existing methods such as SGD, RSG and RES. We also noticed that the phenomenon shown in Table \ref{tab_per_BB_SVM_1} deserves a further investigation to better understand the behavior of BB steps in designing stochastic quasi-Newton methods, and we leave this as a future work.


\bibliographystyle{plain}

\bibliography{All}

\end{document}